\documentclass[twoside,a4paper,reqno,11pt]{amsart} 
\usepackage[top=28mm,right=30mm,bottom=28mm,left=30mm]{geometry}

\usepackage{amsfonts, amsmath, amssymb, mathrsfs, bm, latexsym, stmaryrd, array, hyperref, mathtools, bold-extra}

\renewcommand{\a}{\alpha}

\newcommand{\e}{\varepsilon}
\renewcommand{\l}{\lambda} 

\newcommand{\s}{\sigma}

\renewcommand{\O}{\Omega}

\newcommand{\la}{\langle}
\newcommand{\ra}{\rangle}

\newcommand{\leqs}{\leqslant}

\makeatletter
\newcommand{\imod}[1]{\allowbreak\mkern4mu({\operator@font mod}\,\,#1)}
\makeatother

\newtheorem{theorem}{Theorem} 
\newtheorem*{conj*}{Conjecture}

\newtheorem{thm}{Theorem}[section] 
\newtheorem{prop}[thm]{Proposition} 
\newtheorem{lem}[thm]{Lemma}
\newtheorem{cor}[thm]{Corollary}

\theoremstyle{definition}

\begin{document}

\author{Timothy C. Burness}
\thanks{The authors thank Derek Holt and Martin Liebeck for helpful conversations regarding the content of this paper.}
\address{T.C. Burness, School of Mathematics, University of Bristol, Bristol BS8 1UG, UK}
\email{t.burness@bristol.ac.uk}
 
  \author{Adam R. Thomas}
 \address{A.R. Thomas, Mathematics Institute,
Zeeman Building, University of Warwick, Coventry CV4 7AL, UK}
 \email{adam.r.thomas@warwick.ac.uk}

\title[Extremely primitive affine groups]{A note on extremely primitive affine groups}

\begin{abstract}
Let $G$ be a finite primitive permutation group on a set $\O$ with nontrivial point stabilizer $G_{\a}$. We say that $G$ is extremely primitive if $G_{\a}$ acts primitively on each of its orbits in $\O\setminus \{\a\}$. In earlier work, Mann, Praeger and Seress have proved that every extremely primitive group is either almost simple or of affine type and they have classified the affine groups up to the possibility of at most finitely many exceptions. More recently, the almost simple extremely primitive groups have been completely determined. If one assumes Wall's conjecture on the number of maximal subgroups of almost simple groups, then the results of Mann et al. show that it just remains to eliminate an explicit list of affine groups in order to complete the classification of the extremely primitive groups. Mann et al. have conjectured that none of these affine candidates are extremely primitive and our main result confirms this conjecture.
\end{abstract}

\date{\today}

\maketitle

\section{Introduction}\label{s:intro}

Let $G \leqs {\rm Sym}(\O)$ be a finite primitive permutation group with point stabilizer $H=G_{\a} \neq 1$. We say that $G$ is \emph{extremely primitive} if $H$ acts primitively on each of its orbits in $\O \setminus \{\a\}$. For example, the natural actions of ${\rm Sym}_{n}$ and ${\rm PGL}_{2}(q)$ of degree $n$ and $q+1$, respectively, are extremely primitive. The study of these groups can be traced back to work of Manning \cite{Manning} in the 1920s and they have been the subject of several papers in recent years \cite{BPS, BPS2, BTh_ep, MPS}.

A key theorem of Mann, Praeger and Seress \cite[Theorem 1.1]{MPS} states that every extremely primitive group is either almost simple or affine, and in the same paper they classify all the affine examples up to the possibility of finitely many exceptions. In later work, Burness, Praeger and Seress \cite{BPS, BPS2} determined all the almost simple extremely primitive groups with socle an alternating, classical or sporadic group. The classification for almost simple groups has very recently been completed in \cite{BTh_ep}, where the remaining exceptional groups of Lie type are handled. We refer the reader to \cite[Theorem 4]{BTh_ep} for the list of known extremely primitive groups.

It is conjectured that the list of extremely primitive affine groups presented in \cite{MPS} is complete, so \cite[Theorem 4]{BTh_ep} gives a full classification. To describe the current state of play in more detail, let $G = V{:}H$ be a finite primitive group of affine type, where $V = \mathbb{F}_p^d$ and $p$ is a prime. In \cite{MPS}, the classification is reduced to the case where $p=2$ and $H$ is almost simple (that is, $H$ has a unique minimal normal subgroup $H_0$, which is nonabelian and simple). A basic tool in the analysis of these groups is \cite[Lemma 4.1]{MPS}, which states that $G$ is not extremely primitive if $|\mathcal{M}(H)|<2^{d/2}$, where $\mathcal{M}(H)$ is the set of maximal subgroups of $H$. If $H$ is a sufficiently large almost simple group, then a theorem of Liebeck and Shalev \cite{LSh96} gives $|\mathcal{M}(H)|<|H|^{8/5}$ and by playing this off against known bounds on the dimensions of irreducible modules for almost simple groups, Mann, Praeger and Seress prove that their list of extremely primitive affine groups is complete up to at most finitely many exceptions. 

If one assumes that the stronger bound $|\mathcal{M}(H)|<|H|$ holds, as predicted by a well known (but still open) conjecture of G.E. Wall \cite{Wall}, then \cite[Theorem 4.8]{MPS} states that the classification of the extremely primitive affine groups (and therefore all extremely primitive groups, given \cite{BPS, BPS2, BTh_ep}) is complete up to determining the status of the groups recorded in Table \ref{tab:cases}. With the exception of the case in the first row, $H_0$ is a simple group of Lie type over $\mathbb{F}_2$ and $V = L(\l)$ is a $2$-restricted irreducible module for $H_0$ with highest weight $\l$. In the table, we express $\l$ in terms of a set of fundamental dominant weights $\l_1, \ldots, \l_r$ for $H_0$, where $r$ is the untwisted Lie rank of $H_0$ and the weights are labelled in the usual way (see \cite{Bou}). Notice that the highest weights in the table are listed up to graph automorphisms of $H_0$.

In \cite{MPS}, Mann, Praeger and Seress conjecture that none of the candidates in Table \ref{tab:cases} are extremely primitive. Our main result confirms this conjecture, thereby completing the classification of the extremely primitive groups (modulo Wall's conjecture for almost simple groups).

\begin{theorem}\label{t:main}
Let $G = V{:}H$ be a primitive permutation group of affine type as in Table
\ref{tab:cases}, where $V = \mathbb{F}_2^d$ and $H$ is almost simple with socle $H_0$. Then $G$ is not extremely primitive.
\end{theorem}

\begin{table}
\[
\begin{array}{cll} \hline
d & H_0 & V \\ \hline
40 & {\rm PSp}_{4}(9) & \mbox{Weil representation} \\
& {\rm L}_{5}(2) & \mbox{$L(\l_1+\l_2)$ or $L(\l_1+\l_3)$} \\
48 & {\rm Sp}_{8}(2) & L(\l_3) \\
& \O_{8}^{\pm}(2) & L(\l_1+\l_3) \\
64 & {\rm Sp}_{12}(2) & L(\l_2) \\
70 & {\rm L}_{8}(2),\, {\rm U}_{8}(2) & L(\l_4) \\
100 & {\rm Sp}_{10}(2) & L(\l_3) \\
126 & {\rm L}_{9}(2) & L(\l_4) \\
\binom{k}{3} & {\rm L}_{k}(2),\, 7 \leqs k \leqs 14 & L(\l_3) \\
2^k & {\rm Sp}_{2k}(2),\, 5 \leqs k \leqs 8 & L(\l_k) \\
 & \O_{2k+2}^{+}(2),\, 5 \leqs k \leqs 8 & L(\l_k) \\ \hline
27 & E_6(2) & L(\l_1) \\
56 & E_7(2) & L(\l_1) \\
78 & E_6(2),\, {}^2E_6(2) & L(\l_2)  \\
132 & E_7(2) & L(\l_7) \\
248 & E_8(2) &  L(\l_1) \\ \hline
\end{array}
\]
\caption{The extremely primitive candidates in \cite[Table 2]{MPS}}
\label{tab:cases}
\end{table}

\section{Preliminaries}\label{s:prel}

Let $G = V{:}H$ be a primitive permutation group of affine type as in Table
\ref{tab:cases}, where $V = \mathbb{F}_2^d$ and $H$ is almost simple with socle $H_0$. Let $\mathcal{M}(H)$ be the set of maximal subgroups of $H$. For $M \in \mathcal{M}(H)$, we define
\[
{\rm fix}(M) = \{ v \in V \,:\, \mbox{$v^x = v$ for all $x \in M$}\} = \bigcap_{x \in M}C_V(x),
\]
the fixed point space of $M$ on $V$. Note that $\dim {\rm fix}(M) \leqs \lfloor d/2 \rfloor$ for all $M \in \mathcal{M}(H)$ (since the primitivity of $G$ implies that $H = \la M, M^h \ra$ acts irreducibly on $V$ for every conjugate $M^h \ne M$). Set 
\begin{equation}\label{e:fH}
f(H) = \sum_{M \in \mathcal{M}(H)} (|{\rm fix}(M)|-1).
\end{equation}
The following result is \cite[Lemma 4.1]{MPS}.

\begin{lem}\label{l:bound}
We have $f(H) \leqs 2^d-1$, with equality if and only if $G$ is extremely primitive.
\end{lem}

\begin{cor}\label{c:basic}
If $(2^{\lfloor d/2 \rfloor}-1)\cdot |\mathcal{M}(H)| < 2^d-1$, then $G$ is not extremely primitive.
\end{cor}

In view of Corollary \ref{c:basic}, we are interested in computing $|\mathcal{M}(H)|$ for all the relevant groups $H$ in Table \ref{tab:cases}. A complete classification of the maximal subgroups of $E_8(2)$ up to conjugacy is currently out of reach, but we can calculate $|\mathcal{M}(H)|$ in all the remaining cases. We will need the following result in the proof of Theorem \ref{t:main}.

\begin{prop}\label{p:maximals}
Let $H$ be an almost simple group with socle $H_0$, where $H_0$ is one of the groups recorded in the first column of Table \ref{tab:max}. 
Then $|\mathcal{M}(H)| \leqs \a(H_0)$, where $\a(H_0)$ is given in the second column of Table \ref{tab:max}.
\end{prop}

\begin{proof}
In each case we use {\sc Magma} \cite{Magma} to construct a set of representatives of the conjugacy classes of maximal subgroups of $H$. Typically, we do this by using the command \texttt{AutomorphismGroupSimpleGroup} to construct ${\rm Aut}(H_0)$ as a permutation group and we then identify $H$ as a subgroup of ${\rm Aut}(H_0)$ (in every case, $H$ is either $H_0$, ${\rm Aut}(H_0)$ or a maximal subgroup of ${\rm Aut}(H_0)$). We then use \texttt{MaximalSubgroups} to construct representatives of the classes of maximal subgroups of $H$ and we compute $|\mathcal{M}(H)|$ by summing the indices $|H:N_H(M)|$ for each representative $M$. The number $\a(H_0)$ presented in the table is the maximum value of $|\mathcal{M}(H)|$ as we range over all the almost simple groups $H$ with socle $H_0$.

For $H_0 = {\rm Sp}_{16}(2)$ and $\O_{16}^{+}(2)$ the command \texttt{MaximalSubgroups} is ineffective and so a slightly modified approach is required. The basic method is identical, but in these cases we use \texttt{ClassicalMaximals} to construct a set of representatives of the classes of maximal subgroups of $H$, combined with \texttt{LMGIndex} to compute the indices.
\end{proof}

\begin{table}
\[
\begin{array}{ll} \hline
H_0 & \a(H_0) \\ \hline
{\rm PSp}_{4}(9) & 612624 \\ 
{\rm L}_{5}(2) & 76479 \\
{\rm L}_{8}(2) & 10845467135 \\
\O_{8}^{+}(2) & 521610 \\
\O_{8}^{-}(2) & 1248652 \\
{\rm L}_{9}(2) & 18204373477974477121 \\
{\rm L}_{10}(2) & 935073229364399584692947 \\
{\rm Sp}_{10}(2) & 151633922695 \\
{\rm L}_{11}(2) & 34118520289259566683898930411622 \\
{\rm L}_{12}(2) & 1902312438544124209061463900701007697 \\
{\rm L}_{13}(2) & 2029650642403883210310724134646854692111646099 \\
{\rm L}_{14}(2) & 15558931967070790255179574153313525787726469722554 \\
{\rm Sp}_{16}(2) & 9309048668836568191706512832 \\ 
\O_{16}^{+}(2) & 431792492092675316700367254 \\ \hline
\end{array}
\]
\caption{The bounds $|\mathcal{M}(H)| \leqs \a(H_0)$ in Proposition \ref{p:maximals}}
\label{tab:max}
\end{table}

We will also need to compute $|\mathcal{M}(H)|$ for $H = \O_{18}^{+}(2)$. By Aschbacher's theorem \cite{Asch} on the subgroup structure of the finite classical groups, each maximal subgroup $M$ of $H$ is either \emph{geometric}, in which case the possibilities for $M$ are determined up to conjugacy in \cite{KL}, or $M$ is \emph{non-geometric}, which means that $M$ is an irreducibly embedded almost simple subgroup. 

\begin{prop}\label{p:o18}
Every maximal subgroup of $H = \O_{18}^{+}(2)$ is geometric.
\end{prop}

\begin{proof}
Let $W$ be the natural module for $H$ and suppose $M$ is a non-geometric maximal subgroup of $H$ with socle $M_0$. By definition, $W$ is an absolutely irreducible module for $M_0$ over $\mathbb{F}_2$. There are two cases to consider, according to whether or not $M_0$ is in ${\rm Lie}(2)$, where ${\rm Lie}(2)$ is the set of simple groups of Lie type in characteristic $2$.

First assume $M_0 \in {\rm Lie}(2)$. By inspecting L\"{u}beck \cite{Lub}, we quickly deduce that the only possibilities for $M_0$ are ${\rm L}_{18}(2)$, ${\rm Sp}_{18}(2)$ and $\O_{18}^{\pm}(2)$ (indeed, these are the only simple groups of Lie type in even characteristic with an $18$-dimensional absolutely irreducible representation over $\mathbb{F}_2$). Plainly, none of these possibilities can arise.

Now assume $M_0 \not\in {\rm Lie}(2)$. Here we turn to the work of Hiss and Malle \cite{HissMalle}, which records all the absolutely irreducible representations of finite quasisimple groups up to dimension $250$, excluding representations in the defining characteristic. In addition, information on the corresponding Frobenius-Schur indicators and fields of definition is also provided. By inspecting \cite{HissMalle}, we see that the only possibilities for $M_0$ are the alternating groups $\text{Alt}_{19}$ and $\text{Alt}_{20}$ (note that ${\rm L}_{2}(19)$ does have an $18$-dimensional absolutely irreducible representation in even characteristic with indicator $+1$, but this is defined over $\mathbb{F}_4$, rather than $\mathbb{F}_2$). For $M_0 = \text{Alt}_{19}$ and $\text{Alt}_{20}$, the relevant representation is afforded by the fully deleted permutation module over $\mathbb{F}_2$. However, we have $\text{Alt}_{19} < \text{Alt}_{20} < \Omega^-_{18}(2)$ (see \cite[p.187]{KL}, for example) and so this  representation does not embed $M_0$ in $H$. 
\end{proof}

\begin{cor}\label{c:o18}
If $H = \O_{18}^{+}(2)$, then $|\mathcal{M}(H)| = 115583493125204258236922964476027$.
\end{cor}

\begin{proof}
In view of Proposition \ref{p:o18}, this is an entirely straightforward  computation using the \texttt{ClassicalMaximals} command in {\sc Magma} \cite{Magma}, which returns a set of conjugacy class representatives of the geometric maximal subgroups of $H$. 
\end{proof}

\section{Proof of Theorem \ref{t:main}: $H$ classical}

We begin the proof of Theorem \ref{t:main} by handling the cases where $H_0$ is a classical group. We first observe that several groups can be immediately eliminated by combining Corollary \ref{c:basic} with Proposition \ref{p:maximals}.

\begin{prop}\label{p:class1}
Let $G = V{:}H$ be an affine group in Table \ref{tab:cases} such that $H_0$ is one of
\[
{\rm PSp}_{4}(9), \, {\rm L}_{5}(2), \,  {\rm L}_{8}(2), \, \O_{8}^{\pm}(2), \, {\rm Sp}_{10}(2),\, {\rm L}_{14}(2),\,  {\rm Sp}_{16}(2), \, \O_{18}^{+}(2).
\]
Then $G$ is not extremely primitive.
\end{prop}

\begin{proof}
For $H_0 \ne \O_{18}^{+}(2)$ we use the upper bound $|\mathcal{M}(H)| \leqs \a(H_0)$ in Proposition \ref{p:maximals} to verify the bound in Corollary \ref{c:basic}. For example, if $H_0 = {\rm PSp}_{4}(9)$, then $d=40$ and $|\mathcal{M}(H)| \leqs 612624$, which yields
\[
(2^{20}-1) \cdot |\mathcal{M}(H)| < 2^{40}-1
\]
as required. Similarly, if $H_0 = \O_{18}^{+}(2)$ then $V = L(\l_9)$ and thus $H = \O_{18}^{+}(2)$ since the highest weight of $V$ is not stable under a graph automorphism of $H_0$. Therefore, Corollary \ref{c:o18} gives $|\mathcal{M}(H)|$ and the result follows as before.
\end{proof}

To complete the proof of Theorem \ref{t:main} for the cases with $H_0$ classical, we may assume that one of the following holds:
\begin{itemize}\addtolength{\itemsep}{0.2\baselineskip}
\item[{\rm (a)}] $(H_0,V) = ({\rm Sp}_{8}(2),L(\l_3))$, $({\rm U}_{8}(2),L(\l_4))$,  $({\rm L}_{9}(2),L(\l_4))$ or $({\rm Sp}_{12}(2),L(\l_2))$.
\item[{\rm (b)}] $H_0 = {\rm L}_{k}(2)$, $7 \leqs k \leqs 13$ and $V = L(\l_3)$.
\item[{\rm (c)}] $H_0 = {\rm Sp}_{2k}(2)$ or $\O_{2k+2}^{+}(2)$, $5 \leqs k \leqs 7$ and $V=L(\l_k)$.  
\end{itemize}
We will handle each of these cases in turn, referring to the labels (a), (b) and (c).

\begin{prop}\label{p:burnside}
If $G = V{:}H$ is an affine group in (a), then $G$ is not extremely primitive.
\end{prop}

\begin{proof}
In each of these cases, we use {\sc Magma} to construct the module $V$ and a set of representatives for the conjugacy classes of maximal subgroups of $H$. To construct $V$, we use the command \texttt{IrreducibleModulesBurnside}, with the optional \texttt{DimLim} parameter equal to $2000$. Apart from the case $H = {\rm Sp}_{12}(2)$, we note that $H$ has a unique $d$-dimensional irreducible module over $\mathbb{F}_2$, up to graph automorphisms. The group $H = {\rm Sp}_{12}(2)$ has two $64$-dimensional irreducible modules, namely $L(\l_2)$ and the spin module $L(\l_6)$, and they can be distinguished by considering their restrictions to a subgroup $\O_{12}^{+}(2)$ of $H$. Indeed, the restriction of $L(\l_2)$ is irreducible, while the restriction of $L(\l_6)$ is reducible (the composition factors are the two $32$-dimensional spin modules for $\O_{12}^{+}(2)$).

Then for each maximal subgroup $M$ of $H$, we compute the $1$-eigenspace $C_V(x)$, where $x$ runs through a set of generators $X$ for $M$. Since ${\rm fix}(M) = \bigcap_{x \in X}C_V(x)$, this allows us to compute $f(H)$ precisely (see \eqref{e:fH}):
\[
f({\rm Sp}_{8}(2)) = 11475,\; f({\rm U}_{8}(2)) = f({\rm U}_{8}(2).2) = 3923366139,
\]
\[
f({\rm L}_{9}(2)) = 3309747, \; f({\rm Sp}_{12}(2)) = 6102339243.
\] 
(Note that if $H_0 = {\rm L}_{9}(2)$ and $V = L(\l_4)$, then the highest weight of $V$ is not invariant under a graph automorphism of $H_0$, so $H \ne {\rm L}_{9}(2).2$.) 

In each case, it is now routine to verify the bound $f(H)< 2^d-1$. By Lemma \ref{l:bound}, this implies that $G$ is not extremely primitive.
\end{proof}

\begin{prop}\label{p:wedge}
If $G = V{:}H$ is an affine group in (b), then $G$ is not extremely primitive.
\end{prop}

\begin{proof}
Here $H_0 = L_{k}(2)$ and $V = L(\l_3) = \Lambda^3W$, where $7 \leqs k \leqs 13$ and $W$ is the natural module for $H_0$. Since the highest weight $\l_3$ is not invariant under a graph automorphism of $H_0$, it follows that $H = {\rm L}_{k}(2)$. The cases with $7 \leqs k \leqs 10$ can be handled by proceeding as in the proof of Proposition \ref{p:burnside} and we find that 
\[
f({\rm L}_{7}(2)) = 11811,\; f({\rm L}_{8}(2)) = 97155,\; f({\rm L}_{9}(2)) = 18202348610724300355,
\]
\[
f({\rm L}_{10}(2)) = 413104411638650042899395.
\]
However, we will give a uniform argument for all $8 \leqs k \leqs 13$.

With the aid of {\sc Magma}, it is easy to verify that each maximal subgroup $M$ of $H$ contains an element of order $r \in \{7,11,13\}$. Since $V$ is simply the wedge-cube of $W$, it is straightforward to calculate 
$\dim C_V(x)$ for each element $x \in H$ of order $r$. 

For example, suppose $H = {\rm L}_{8}(2)$ and $x \in H$ has order $7$. Now $H$ has three conjugacy classes of elements of order $7$, with representatives
\[
x_1 = [I_5, \omega, \omega^2,\omega^4],\; x_2 = [I_5, \omega^3, \omega^5,\omega^6],\; x_3 = [I_2, \omega, \omega^2,\omega^3, \omega^4, \omega^5,\omega^6],
\]
where $\omega \in \mathbb{F}_{2^3}$ is a primitive $7$-th root of unity (see \cite[Section 3.2]{BG}, for example). Notice that we are viewing the conjugacy class representatives as diagonal matrices in ${\rm SL}_{8}(8)$, which is convenient for computing their $1$-eigenspaces on $V$. By considering the eigenvalues of $x_i$ on $W$, it is easy to show that $\dim C_V(x_1) = \dim C_V(x_2) = \binom{5}{3}+1 = 11$ and $\dim C_V(x_3) = 8$.

In this way, we find that $\dim C_V(x) \leqs \binom{k-3}{3}+1$ for all $x \in H$ of order $r \in \{7,11,13\}$, with equality if and only if $r = 7$ and $\dim C_W(x) = k-3$. Therefore
\[
f(H) \leqs (2^{\binom{k-3}{3}+1}-1)\cdot |\mathcal{M}(H)|
\]
and by applying the bound $|\mathcal{M}(H)| \leqs \a(H_0)$ in Proposition \ref{p:maximals}, we deduce that $f(H) < 2^d-1$ for $8 \leqs k \leqs 13$.
\end{proof}

\begin{prop}\label{p:spin1}
If $G = V{:}H$ is an affine group in (c), then $G$ is not extremely primitive.
\end{prop}

\begin{proof}
Here $H = {\rm Sp}_{2k}(2)$ or $\O_{2k+2}^{+}(2)$ and $V = L(\l_k)$ is a spin module with $5 \leqs k \leqs 7$ (note that if $H_0 = \O_{2k+2}^{+}(2)$ then the highest weight of $V$ is not stable under graph automorphisms, so $H = H_0$). 

First assume $H = {\rm Sp}_{2k}(2)$. For $k=5$ and $6$ we can compute $f(H)$ by proceeding as in the proof of Proposition \ref{p:burnside}; we get 
\[
f({\rm Sp}_{10}(2)) = 75735,\; f({\rm Sp}_{12}(2)) = 4922775
\]
and the result follows from Lemma \ref{l:bound}. Now assume $H = {\rm Sp}_{14}(2)$ and fix a maximal field extension subgroup $L = {\rm Sp}_{2}(2^7).7$. Write $\mathcal{M}(H) = \mathcal{M}_1 \cup \mathcal{M}_2$, where $\mathcal{M}_1$ is the set of $H$-conjugates of $L$. Using {\sc Magma} (with the \texttt{ClassicalMaximals} command) we see that 
\[
|\mathcal{M}_1| = 1902762402163023937536000, \;\;  |\mathcal{M}_2| = 407915701794349.
\]

Let $\bar{H} = C_7$ be the ambient simple algebraic group over the algebraic closure $k=\bar{\mathbb{F}}_{2}$ and fix a Steinberg endomorphism of $\bar{H}$ such that $\bar{H}_{\s} = H$. We may choose $\s$ so that $\bar{L}_{\s} = L$, where $\bar{L}$ is a maximal rank subgroup of $\bar{H}$ with connected component $\bar{L}^0 = A_1^7$ (that is, $\bar{L}^0$ is a central product of $7$ copies of $A_1 = {\rm SL}_{2}(k)$). Now the restriction of the spin module for $\bar{H}$ to $\bar{L}^0$ is the tensor product of the natural modules for the $A_1$ factors. In particular, the restriction is irreducible and we deduce that $L$ acts irreducibly on $V$. Therefore, ${\rm fix}(M) = 0$ for all $M \in \mathcal{M}_1$ and thus
\[
f(H)  = \sum_{M \in \mathcal{M}_2}(|{\rm fix}(M)|-1) \leqs (2^{64}-1)\cdot |\mathcal{M}_2| < 2^{128}-1
\]
since the trivial bound $\dim {\rm fix}(M) \leqs 2^{64}$ holds for all $M \in \mathcal{M}_2$. By applying 
Lemma \ref{l:bound}, we conclude that $G$ is not extremely primitive.

For the remainder of the proof, we may assume that $H = \O_{2k+2}^{+}(2)$ with $5 \leqs k \leqs 7$. The case $k=5$ can be handled as in Proposition \ref{p:burnside} (in order to construct $V$ using the command \texttt{IrreducibleModulesBurnside}, we need to set \texttt{DimLim} equal to $10000$) and we get $f(H) = 1240917975<2^{32}-1$. 

Next assume $k=6$, so $H = \O_{14}^{+}(2)$. Let $W$ be the natural module for $H$ and let $\mathcal{M}_1$ be the set of maximal subgroups of $H$ of one of the following types:
\[
P_1, \; P_3, \; P_4, \; P_7 \, \mbox{(both classes)}, \; {\rm O}_{8}^{+}(2) \times {\rm O}_{6}^{+}(2),\; {\rm Sp}_{12}(2),
\]
where $P_m$ denotes the stabilizer in $H$ of a totally singular $m$-dimensional subspace of $W$. In addition, let $\mathcal{M}_2$ be the remaining reducible maximal subgroups of $H$ and write $\mathcal{M}_3$ for the set of $H$-conjugates of a fixed irreducible subgroup $L = {\rm L}_{2}(13)$. With the aid of {\sc Magma}, it is easy to check that 
\[
|\mathcal{M}_1| = 240862567876011, \; |\mathcal{M}_2| = 166862538433514
\]
and $\mathcal{M}(H) = \mathcal{M}_1 \cup \mathcal{M}_2 \cup \mathcal{M}_3$.

Now, $|\mathcal{M}_3| = |H:L|>2^{64}-1$, so Lemma \ref{l:bound} implies that ${\rm fix}(L)=0$ and thus
\[
f(H) = \sum_{M \in \mathcal{M}_1} (|{\rm fix}(M)|-1) + \sum_{M \in \mathcal{M}_2} (|{\rm fix}(M)|-1).
\]
There are two conjugacy classes of elements of order $7$ in $H$ and both classes have representatives in a reducible subgroup $K = {\rm Sp}_{12}(2)$. In order to compute the $1$-eigenspaces of these elements on $V$, it is convenient to work in the corresponding algebraic groups over 
$\bar{\mathbb{F}}_2$, so write $\bar{H} = D_7$ and $\bar{K} = B_6$. Then the two classes of order $7$ in $H$ are represented by the elements 
\[
x_1 = [I_6, \omega, \omega^2, \omega^3, \omega^4,\omega^5,\omega^6],\;\; x_2 = [\omega I_2, \omega^2I_2, \omega^3I_2, \omega^4I_2,\omega^5I_2,\omega^6I_2]
\]
in $\bar{K}$, where $\omega \in \bar{\mathbb{F}}_2$ is a primitive $7$-th root of unity. Let $\bar{V} = V \otimes \bar{\mathbb{F}}_2$ be the spin module for $\bar{H}$, which remains irreducible on restriction to a maximal rank subgroup $\bar{J} = A_1^6$ of $\bar{K}$; the restriction is the tensor product of the natural $2$-dimensional modules for the $A_1$ factors of $\bar{J}$. This allows us to compute $\dim C_V(x_i)$ very easily. For example, the action of $x_1$ on $\bar{V}$ is given by  
\[
[I_2] \otimes [I_2] \otimes [I_2] \otimes [\omega, \omega^6] \otimes [\omega^2, \omega^5] \otimes [\omega^3, \omega^4] = [I_{16},\omega I_8, \ldots, \omega^6I_8]
\]
and thus $\dim C_V(x_1) = 16$. Similarly, $\dim C_V(x_2) = 10$.

Using {\sc Magma}, it is easy to check that each subgroup $M \in \mathcal{M}_1 \cup \mathcal{M}_2$ contains an element of order $7$. Moreover, each $M \in \mathcal{M}_1$ contains a conjugate of $x_2$. It follows that
\[
f(H) \leqs (2^{10}-1) \cdot |\mathcal{M}_1| + (2^{16}-1) \cdot |\mathcal{M}_2| = 11181738863177499243 < 2^{64}-1
\]
and we deduce that $G$ is not extremely primitive.

Finally, let us assume $k=7$, so $H = \O_{16}^{+}(2)$. Let $W$ be the natural module for $H$. We claim that $\dim {\rm fix}(M) \leqs 32$ for all $M \in \mathcal{M}(H)$. Given the claim, together with the bound on $|\mathcal{M}(H)|$ in Proposition \ref{p:maximals}, it follows that
\[
f(H) \leqs (2^{32}-1) \cdot |\mathcal{M}(H)| < 2^{128}-1
\] 
and the proof is complete. So it remains to justify the claim.

As in the previous case, $H$ has two conjugacy classes of elements of order $7$, represented by elements $x_1,x_2$ in a reducible subgroup ${\rm Sp}_{14}(2)$, where $\dim C_W(x_1) = 10$ and $\dim C_W(x_2) = 4$. By arguing as in the previous case, we calculate that $\dim C_V(x_1) = 32$ and $\dim C_V(x_2) = 20$. Similarly, there is an element $x_3 \in {\rm Sp}_{14}(2)$ of order $5$ with $C_W(x_3) = 4$ and $\dim C_V(x_3) = 24$. By constructing representatives of the maximal subgroups of $H$ in {\sc Magma} (using \texttt{ClassicalMaximals}), it is straightforward to check that each $M \in \mathcal{M}(H)$ contains an element conjugate to either $x_1$, $x_2$ or $x_3$ (indeed, the order of every maximal subgroup of $H$ is divisible by $7$, apart from the subgroups of type ${\rm O}_{4}^{-}(2) \wr {\rm Sym}_4$). This justifies the claim and the proof of the proposition is complete.
\end{proof}

\section{Proof of Theorem \ref{t:main}: $H$ exceptional}

In this final section we complete the proof of Theorem \ref{t:main} by handling the remaining cases in Table \ref{tab:cases} with $H_0$ an exceptional group of Lie type.

\begin{prop}\label{p:exceptional}
Let $G = V{:}H$ be an affine group in Table \ref{tab:cases} with $H_0$ an exceptional group of Lie type. Then $G$ is not extremely primitive.
\end{prop}

\begin{proof}
In each case we will demonstrate the existence of a nonzero vector $v \in V$ such that the point stabilizer $C_H(v)$ is a non-maximal subgroup of $H$.

First assume $(d,H_0) = (27, E_6(2))$, so $V = L(\l_1)$ or $L(\l_6)$ is one of the minimal modules for $H_0$. Since the highest weight of $V$ is not invariant under a graph automorphism of $H_0$, it follows that $H = H_0$. 
By \cite[p.467]{CC}, there exists $v \in V$ such that $C_{H}(v) = 2^{16}.\text{Sp}_8(2)$, which is a non-maximal subgroup of $H$ by \cite{KW}. Similarly, if $(d,H_0) = (56, E_7(2))$ then $H = H_0$ and \cite[Lemma 4.3]{LieSaxl} implies that there exists $v \in V$ with $C_{H}(v) = 2^{26}.F_4(2)$. By inspecting \cite{BBR}, we see that $C_H(v)$ is non-maximal in $H$.

In the final three cases, $V$ is the unique nontrivial composition factor of the adjoint module for $H_0$ (note that the adjoint module is irreducible when $H_0 = E_6^{\e}(2)$ or $E_8(2)$, but there are two composition factors if $H_0 = E_7(2)$). Write $H_0 = (\bar{H}_{\s})'$, where $\bar{H}$ is a simple algebraic group of adjoint type over $\bar{\mathbb{F}}_2$ and $\s$ is an appropriate Steinberg endomorphism of $\bar{H}$. Let $\mathcal{L}(\bar{H})$ be the adjoint module for $\bar{H}$, which is simply the Lie algebra of $\bar{H}$ equipped with the adjoint action of $\bar{H}$, and note that we may view $V$ as a subset of $\mathcal{L}(\bar{H})$. Recall that the orbits for the action of $\bar{H}$ on the set of nilpotent elements of $\mathcal{L}(\bar{H})$ are called nilpotent orbits. 

If $\bar{H}=E_7$, then the adjoint module $\mathcal{L}(\bar{H})$ has a unique nontrivial composition factor $\bar{V}$ and it will be important to note that every nilpotent orbit of $\bar{H}$ has a representative in $\bar{V}$. Since we are working in even characteristic, we may assume that $\bar{H}$ is simply connected and we see that $\mathcal{L}(\bar{H})$ has a $1$-dimensional centre, which is generated by a semisimple element. Therefore the $132$-dimensional quotient $\bar{V}$ contains a representative of every nilpotent orbit as claimed.

It follows from \cite[Section 1]{HS} that every nilpotent orbit on $\mathcal{L}(\bar{H})$ has a representative defined over the prime field $\mathbb{F}_2$. Therefore, in every case we may choose $v \in V$ to be a representative of the nilpotent orbit labelled $A_1^2$ in \cite[Tables 22.1.1--22.1.3]{LieS}, which also gives the structure of the stabilizer $C_{\bar{H}}(v)$. Moreover, $C_{\bar{H}}(v)$ is $\s$-stable because it is the only stabilizer of a nilpotent element with its given dimension. Therefore, $C_{\bar{H}_\s}(v) = (C_{\bar{H}}(v))_\s$. 

First assume $\bar{H} = E_7$ or $E_8$. Here $H = \bar{H}_\s$, so 
$C_{H}(v) = (C_{\bar{H}}(v))_\s$ and by inspecting \cite{LieS} we see that 
\[
C_{\bar{H}}(v) = \left\{\begin{array}{ll}
U_{42} B_4 A_1 & \mbox{if $\bar{H} = E_7$} \\
U_{78} B_6 & \mbox{if $\bar{H} = E_8$,}
\end{array}\right.
\]
where $U_i$ denotes a connected unipotent algebraic group of dimension $i$. In particular, $C_{\bar{H}}(v)$ is a proper subgroup of a $\s$-stable maximal parabolic subgroup of $\bar{H}$, whence $C_{H}(v)$ is non-maximal in $H$. 

Finally, let us assume $\bar{H} = E_6$ and $H_0 = E_6^{\e}(2)$. Here we get 
\[
|C_{\bar{H}_\s}(v)| = 2^{24}|{\rm Sp}_{6}(2)|.(2-\e)
\]
and thus $|C_{H_0}(v)| = 2^{24}|{\rm Sp}_{6}(2)|$. By appealing to \cite{KW, Wil2}, we see there is no maximal subgroup $M$ of $H$ such that $|M \cap H_0| = |C_{H_0}(v)|$. Therefore $C_{H}(v)$ is non-maximal and the proof is complete.
\end{proof}

\end{document}